\documentclass[11pt,reqno]{amsart}
\usepackage{latexsym}
\usepackage{amsmath}
\usepackage{amssymb, amscd, amsthm}
\usepackage[all]{xy}
\usepackage{multirow}

\newtheorem{theorem}{{Theorem}}[section]
\newtheorem*{theorema}{{Theorem A}}
\newtheorem*{theoremb}{{Theorem B}}
\newtheorem*{theoremc}{{Theorem C}}

\newtheorem{corollary}[theorem]{{Corollary}}
\newtheorem{lemma}[theorem]{{Lemma}}
\newtheorem{proposition}[theorem]{Proposition}

\theoremstyle{remark}
\newtheorem*{remark}{Remark}

\newcommand{\F}{\ensuremath{\mathbb F}}
\newcommand{\Z}{\ensuremath{\mathbb Z}}
\newcommand{\N}{\ensuremath{\mathbb N}}

 \newcommand{\bl}{\mathbf{\lambda}}
 \newcommand{\bk}{\mathbf{k}}
\newcommand{\bx}{\mathbf{x}}
\newcommand{\bb}{\mathbf{b}}

\begin{document}

\title{Counting the solutions of $\lambda_1 x_1^{k_1}+\cdots +\lambda_t x_t^{k_t}\equiv c\bmod{n}$}
\author{Songsong Li}\author{Yi Ouyang}
\address{Wu Wen-Tsun Key Laboratory of Mathematics,  School of Mathematical Sciences, University of Science and Technology of China, Hefei, Anhui 230026, China}
\email{songsli@mail.ustc.edu.cn, yiouyang@ustc.edu.cn}

\date{}
 \maketitle

\begin{abstract}
 Given a polynomial $Q(x_1,\cdots, x_t)=\lambda_1 x_1^{k_1}+\cdots +\lambda_t x_t^{k_t}$, for every $c\in \Z$ and $n\geq 2$, we study the number of solutions $N_J(Q;c,n)$ of the congruence equation $Q(x_1,\cdots, x_t)\equiv c\bmod{n}$ in $(\Z/n\Z)^t$ such that  $x_i\in (\Z/n\Z)^\times$ for $i\in J\subseteq I= \{1,\cdots, t\}$. We deduce formulas and an algorithm to study $N_J(Q; c,p^a)$ for $p$ any prime number and $a\geq 1$ any integer. As consequences of our main results, we completely solve: the counting problem of $Q(x_i)=\sum\limits_{i\in I}\lambda_i x_i$ for any prime $p$ and any subset $J$ of $I$;  the  counting problem of $Q(x_i)=\sum\limits_{i\in I}\lambda_i x^2_i$ in the case $t=2$ for any $p$ and $J$, and the case $t$ general for any $p$ and $J$ satisfying $\min\{v_p(\lambda_i)\mid i\in I\}=\min\{v_p(\lambda_i)\mid i\in J\}$; the  counting problem of $Q(x_i)=\sum\limits_{i\in I}\lambda_i x^k_i$ in the case $t=2$ for any $p\nmid k$ and any $J$, and in the case $t$ general  for any $p\nmid k$ and $J$ satisfying $\min\{v_p(\lambda_i)\mid i\in I\}=\min\{v_p(\lambda_i)\mid i\in J\}$.
\end{abstract}

\section{Introduction and Main results}
\subsection{Introduction} \label{sec_intro}
Given a polynomial
 \[ Q(x_1,\cdots, x_t)=\lambda_1 x_1^{k_1}+\cdots +\lambda_t x_t^{k_t}\in \Z[x_1,\cdots, x_t]. \]
Let $\bl=(\lambda_1,\cdots, \lambda_t)\in (\Z-\{0\})^t$ and $\bk=(k_1,\cdots, k_t)\in \Z_{\geq 1}^t$. For any $c\in \Z$ and $n\geq 2$, and for a subset $J$ of $I=\{1,\cdots, t\}$, denote by $\Gamma_{J}(c, n)=\Gamma_J(Q;c,n)=\Gamma_{J}(\bl,\bk;  c, n)$  the set of solutions $(x_1,\cdots, x_t)$ of the congruence equation
 \[ Q(x_1,\cdots, x_t) \equiv c\bmod{n} \]
such that $x_j\in (\Z/n\Z)^\times$ for $j\in J$, and by $N_J(Q; c,n)$ the cardinality of $\Gamma_J(Q; c,n)$. In particular, write $\Gamma$, $N$, $\Gamma^*$ and $N^*$ for $\Gamma_{\emptyset}$, $N_{\emptyset}$, $\Gamma_{I}$  and $N_{I}$ respectively.
The problem to determine $N_{J}(Q;  c, n)$   has been studied by many authors extensively in various special cases:
\begin{itemize}
\item[(i)] The linear  case $\bk=(1,\cdots, 1)$.  For $J=I$, this is a problem proposed by H. Rademacher\cite{rade} in 1925 and answered by A. Brauer\cite{brau} in 1926, and recovered by many authors later from time to time. For $J=I$ and  $\bl=(\lambda_i)$ where the $\lambda_i$'s are divisors of $n$, this is the work of Sun and Yang\cite{sun-yang} in 2014.

\item[(ii)] The quadratic case $\bk=(2,\cdots, 2)$. For  $J=\emptyset$, this is studied in the work of T\'oth\cite{toth} in 2014. For $t=2$ and $\bl=(1,1)$, this is the work of Yang and Tang\cite{yang-tang} in 2015. For $t=2$, $\bl$ arbitrary and $J=I$, this is the work of Sun and Cheng~\cite{sun-cheng} in 2016. For general $t$,  $\bl=(1,\cdots,1)$ and $J=I$, this is the recent work of Mollahajiaghaei~\cite{mol}. See also \cite{cal} for more development.

\item[(iii)] The case $t=2$, $\bl=(1,1)$ and $\bk=(k, k)$. Partial results were obtained by Deaconescu and Du~\cite{deac}.
\end{itemize}

\subsection{Notations.}
Before stating our main results,  let us fix the following notations.

In this paper, $p$ is always a prime number and $v_p$ is the $p$-adic valuation, $a$ is always  a positive integer and $I$ is the set $\{1,\cdots, t\}$.

For a set $X$, $\# X$ or $|X|$ means the cardinality of $X$. For two subsets $X$ and $Y$ of the set $U$, the difference set $X- Y$ is the set $\{x\in U\mid x\in X,\ x\notin Y\}$.

For the congruence equation
 \[Q(x_1,\cdots, x_t)=\lambda_1 x_1^{k_1}+\cdots +\lambda_t x_t^{k_t}\equiv c\bmod n,\quad (c\in \Z,\ n\in \Z_{\geq 2}) \]
we call $t$, $\bk$ and $n$ its \emph{dimension},  \emph{degree} and \emph{level}  respectively.

For $J$ a nonempty subset of $I$,
the \emph{depth} $d_{p,J}=d_{p,J}(Q)=d_{p,J}(\bl,\bk)$ of $Q$  at prime $p$ associated to $J$ is defined by
 \[ d_{p,J}=\begin{cases} \min\limits_{i\in J}\{v_p(\lambda_ik_i)+1\}, &\text{if}\ p\ \text{odd};\\ \min\limits_{i\in J}\{v_2(\lambda_ik_i)+2\ \text{if $2\mid k_i$},\ v_2(\lambda_ik_i)+1\ \text{if $2\nmid k_i$}\}, &\text{if}\ p=2. \end{cases} \]
Write $d_p$ for $d_{p,I}$ and call it the depth of  of $Q$ at $p$.

For $J$ a nonempty subset of $I$, let $\bl_J=(\lambda_i)_{i\in J}$, $\bk_J=(k_i)_{i\in J}$ and $Q_J=\sum\limits_{j\in J}\lambda_j x^{k_j}\in \Z[x_j:\ j\in J]$. Set $Q_\emptyset=0$ and
 \[ \N_{\emptyset}(0; c,p^a)=N^*(0; c,p^a)=\begin{cases} 1, & \text{if}\ p^a\mid c;\\ 0, & \text{if}\ p^a\nmid c. \end{cases} \]

If $Q$ and $(\bl,\bk)$  are clear from the context, we may drop them in our notations.

\subsection{Main results}
Suppose $n$ has the  prime decomposition
  \[ n=\prod_{p\mid n} p^{n_p}. \]
By Chinese Remainder Theorem, the set of solutions of  $Q(x_1, \cdots,  x_t) \equiv c\bmod n$ is in one-to-one correspondence with the product set of solutions of the equations $Q(x_1,\cdots, x_t)\equiv c\bmod p^{n_p}$ for primes $p\mid n$. Moreover, $x\in (\Z/n\Z)^\times$ if and only if $x\in (\Z/p^{n_p})^\times$ for all $p\mid n$. Thus for any $J\subseteq I$, we have the product formula
 \begin{equation} N_J(Q; c, n)=\prod_{p\mid n} N_J(Q; c, p^{n_p}).
 \end{equation}
So to compute $N_J(Q; c,n)$, it suffices to study the prime power case $N_J(Q; c,p^a)$. By simple argument (as seen in Proposition~\ref{prop:easy123}(2)), we may reduce $Q$ to the case  $p\nmid \lambda_i$ for some $i\in I$, which we call $Q$ is \emph{reduced at $p$}.

Our first main result is the following theorem (which we call the decomposition formula):

\begin{theorema}  Given the polynomial $Q$. For subsets $J_1\subsetneq J_2\subseteq I$, let
 \[ B_i(J_1,J_2; a)=\begin{cases} \{0\}, &\text{if}\ i\notin J_2-J_1;\\  \{0,\cdots, a\}, & \text{if}\ i\in J_2-J_1,\end{cases} \quad  B(J_1, J_2;a)=\prod\limits_{i=1}^t B_i(J_1,J_2; a).\]
For $\bb\in B(\emptyset, I; a)$, but $\bb\neq (a,\cdots, a)$, let
 \[ J_\bb=\{i\in I\mid b_i<a\},  \quad
  Q_\bb=\sum_{j\in J_\bb} \lambda_j p^{k_j b_j} x^{k_j}, \quad  s(\bb)= \sum\limits_{j\in J_\bb} b_j.\]
If $\bb= (a,\cdots, a)$, let  $J_\bb=\emptyset$, $Q_\bb=0$ and $s(\bb)=0$. Then we have the \emph{decomposition formula}
 \begin{equation} \label{eq:nj} N_{J_1}(Q; c, p^a)=\sum_{\bb\in B(J_1, J_2; a)} p^{-s(\bb)}N_{J_2\cap J_\bb}(Q_\bb; c, p^a).
 \end{equation}
 \end{theorema}

Our next two results are  consequences of the following \emph{lifting formula}
 \begin{equation} N_J(Q; c, p^{a+1})=p^{t-1} N_J(Q; c, p^{a})
 \end{equation}
for $a$ sufficiently large under various assumptions. We shall establish this formula by simple $p$-adic analysis, not by the more complicated exponential sum argument employed by other authors. More precisely, we have

\begin{theoremb} Given the polynomial $Q$, and assume it is reduced at prime $p$. Then

(1) For $a\geq d_{p,J}$ and $c\in \Z$,
 \begin{equation} \label{eq:power1x1j}  N_J(Q; c,p^a)=
 p^{(t-1)(a-d_{p,J})} N_J(Q; c,p^{d_{p,J}}). \end{equation}

 (2) For $a\leq d_p=d_{p,I}$, the map
  \[ \varphi_{a}: (\F_p)^t\rightarrow \Z/p^a\Z,\ (a_1,\cdots, a_t)\mapsto Q(\alpha_1,\cdots, \alpha_t)\bmod p^a, \]
 where $\alpha_i\in \Z$ is any lifting of $a_i\in \F_p$, is well defined. Let $\varphi_{a, J}$ be the restriction of $\varphi_a$ on $\prod\limits_{i\in I-J} \F_p\times \prod\limits_{i\in J} \F^\times_p$, then
 \begin{equation} \label{eq:power112j}  N_J(Q; c,p^a)=
 p^{(a-1)t}\# \varphi^{-1}_{a,J}(c\bmod p^a). \end{equation}
In particular, if $p=2$ and $a\leq d_2$,
 \begin{equation} \# \varphi^{-1}_{a,J}(c\bmod 2^a)=\# \{T\subseteq\{1,\cdots,t\}\mid
 T\supseteq J,\ v_2(\sum\limits_{i\in T}\lambda_i-c)\geq a\}.
 \end{equation}
\end{theoremb}
\begin{theoremc}  Given polynomial $Q$ and prime $p$. Let $f_p=\max\{v_p(k_i)+1\}$ (or $3$ if $p=2$ and $\max\{v_2(k_i)\}=1$). For integer $c\neq 0$, let $c_p$ be the $p$-adic valuation of $c$. Then for any $a\geq 1$,  any $J\subseteq I$ (empty or not), $f\geq f_p$ and any $x\in \Z/p^a\Z$,
 \begin{equation} \label{eq:powerc1}  N_J(Q; c(1+p^{f}x),p^a)=
  N_J(Q; c,p^a). \end{equation}
In particular, for $a\geq c_p+f_p$,
 \begin{equation} \label{eq:powerc1}  N_J(Q; c,p^a)=
 p^{(t-1)(a-c_p-f_p)} N_J(Q; c,p^{c_p+f_p}). \end{equation}
Thus $N_J(Q; c,p^a)$ as $a$ varies is completely determined by $N_J(c, p^a)$ for $a\leq c_p+f_p$.
\end{theoremc}
\begin{remark} For $J=\emptyset$, even if $p\nmid \prod\limits_{i=1}^t k_i$, the formula for $N(Q; 0,p^a)$ is much more complicated. In general we don't always have  $N(0, p^a)= p^{t-1} N(0,  p^{a-1})$  for $a$ sufficiently large. For example, consider $Q(x_1, x_2)=x_1^3+p x_2^3$. Then $N(0,p^{3a})=p^{4a}$, $N(0,p^{3a+1})=p^{4a+1}$ and $N(0,p^{3a+2})=p^{4a+2}$.
\end{remark}

As a consequence of Theorems A, B and C, we will give an algorithm to effectively compute $N_J(Q;c,p^a)$ for all possible $J$, $c$ and $a$ if the prime number $p\nmid \prod\limits_{i=1}^t k_i$. Moreover, except the case $J=\emptyset$ and  $c= 0$, the number of steps to compute $N_J(Q;c,p^a)$ is bounded by a constant independent of $a$.

Using the main theorems and the algorithm, we shall work on the example $Q(x_1,\cdots, x_t)=\lambda_1 x_1^k+\cdots +\lambda_t x^k_t$. We obtain the following results:
 \begin{enumerate}
  \item In the linear case ($k=1$), we solve the counting problem in full generality (cf. \cite{sun-yang}). Namely, for any prime $p$, we completely determine the value of $N_J(Q;c, p^a)$ for arbitrary $J\subseteq I$, $c\in \Z$ and $a\geq 1$. Our result is stated in Theorem~\ref{theo:linear}.

  \item In the quadratic case ($k=2$), for any prime $p$, we completely determine the value of  $N_J(Q;c, p^a)$ for  any $J\subseteq I$ satisfying $\min\{v_p(\lambda_i)\mid i\in I\}=\min\{v_p(\lambda_i)\mid i\in J\}$, and arbitrary $c\in \Z$ and integer $a\geq 1$. In particular, we get the exact formula for $N^*(Q;c,p^a)$ for any $c\in \Z$ and $a\geq 1$. Our result is stated in Theorem~\ref{theo:quad}. This is a vast generalization of Yang-Tang~\cite{yang-tang}, Sun-Cheng~\cite{sun-cheng} and  Mollahajiaghaei~\cite{mol}.

  \item In the general case, for prime $p\nmid k$, we give a more detailed version of our algorithm in Theorem~\ref{theorem:kkk1}. We obtain formulas so that $N_J(Q;c,p^a)$ can be computed in finite steps independent of $a$ except the case $c=0$ and $J=\emptyset$.

  \item We study the case $p\nmid k$ and the dimension $t=2$ in full generality. When $k=2$, $N_J(c, 2^a)$ is also studied in full generality.
 \end{enumerate}
Finally we shall work on the example $Q(x_1, \cdots, x_t)= 9x_1+ 2 x_2^3+ x_3^9$ for $p=3$, which is not covered by our algorithm, but the main theorems are still applicable.

\section{Preliminaries} \label{sec_formulas}

 \subsection{Reduce $Q$ to the reduced case.}
The following fact is obvious:

\begin{proposition} \label{prop:easy123} Consider the number $N_J(Q;  c, p^a)$ for $p$ a prime number and $J\subseteq I$.

(1) (Lowering dimension) If there exists $j\in I$ such that $v_p(\lambda_j)\geq a$,  then
 \begin{equation} \label{eq:aca} N_J(Q;  c, p^a)=\begin{cases} p^{a} N_J(Q_{I-\{j\}}; c, p^a), & \text{if}\ j\notin J;\\ p^{a-1}(p-1)N_{J-\{j\}}(Q_{I-\{j\}}; c, p^a), & \text{if}\ j\in J.
 \end{cases}\end{equation}

(2) (Lowering level) Let $e=\min\{v_p(\lambda_i)\mid i\in I\}$ and $v_p(c)=c_p$.
Then
 \begin{equation} \label{eq:aca2} N_J(Q;  c, p^a)=\begin{cases} p^{te} N_J(Q/p^e;  c/p^e, p^{a-e}), &\ \text{if}\ e\leq \min\{a,c_p\};\\ p^{(at-|J|)}(p-1)^{|J|}, &\ \text{if}\ a\leq \min\{e,c_p\}; \\
 0, &\ \text{if}\ c_p<\min\{e,a\}. \end{cases}
 \end{equation}

(3) (Lowering degree) If one has $v_p(k_i)\geq a$, replace $k_{i}$ by $k_i/p^{v_p(k_i)-a+1}$.
Then the new $k_i$ has $p$-adic valuation $a$ and $N_J(Q;  c, p^a)$ is unchanged.
\end{proposition}
\begin{proof} The only thing needs to prove is (3), which follows from Euler's Theorem that for $x\in (\Z/p^a\Z)^\times$, $x^{p^s}=x^{p^{a-1}}$ for all $s\geq a$, and for $x\in p\Z/p^a\Z$, $x^{p^s}=0$ for all $s\geq a-1$ since $p^{a-1}\geq a$ for any prime $p$ and integer $a\geq 1$.
\end{proof}
Based on Proposition~\ref{prop:easy123}, to compute $N_J(Q;  c, p^a)$, it suffices to consider the case that $\min\{v_p(\lambda_i)\}=0$, $\max\{v_p( \lambda_i),\ v_p(k_i)\mid i=1,\cdots, t\}< a$  and the depth $d_p\leq a$. In particular, we can always assume $p\nmid \lambda_i$ for some $i\in I$.

\subsection{Formulas for $N(Q; c,p)$.}
We  recall the classical formulas for $N(Q; c,p)$. First recall for complex characters $\chi_1,\ \cdots,\ \chi_t$  of the prime field $\F_p$, the \emph{Jacobi sum} $J(\chi_1,\cdots,\chi_t)$ is defined by the formula
  \[
   J(\chi_1,\cdots,\chi_t)=
   \sum_{u_1+\cdots+u_t=1}\chi_1(u_1)\cdots\chi_t(u_t)\]
and the \emph{Jacobi sum} $J_0(\chi_1,\cdots,\chi_t)$ is defined by the formula
  \[   J_0(\chi_1,\cdots,\chi_t)=
   \sum_{u_1+\cdots+u_t=0}\chi_1(u_1)\cdots\chi_t(u_t).
 \]
Then the following theorem is well known:

\begin{theorem} \label{theorem:power3} (1) Suppose $p$ is odd and $\lambda_1\cdots\lambda_t\neq 0\in \F_p$. Then $N(c,p)$, the number of solutions of
 \[ Q(x_1,\cdots, x_t)=\lambda_1 x^{k_1}+\cdots \lambda_t x^{k_t} =c \]
over the prime field $\F_p$, is given by
\begin{equation} \label{eq:jacobi0}
   N(0,p)=p^{t-1}+\sum_{\substack{\chi_i^{k_i}=1,\ \chi_i\neq 1\\ \chi_1\cdots\chi_t= 1}} \chi_1(\lambda_1^{-1})
   \cdots\chi_t(\lambda_t^{-1})J_0(\chi_1, \cdots,\chi_t),
\end{equation}
and
\begin{equation} \label{eq:jacobi}
   N(c,p)=p^{t-1}+\sum_{\substack{\chi_i^{k_i}=1\\ \chi_i\neq 1}} \chi_1 \cdots\chi_t(c)\chi_1(\lambda_1^{-1}) \cdots\chi_t(\lambda_t^{-1})J(\chi_1,\chi_2,\cdots,\chi_t)
\end{equation}
for $c\neq 0$.

(2) If $2\nmid \lambda_i$ for some $i\in I$, then $N(0,2)=N(1,2)=2^{t-1}$.
\end{theorem}
\begin{proof} Part (1) follows from Theorem 5 in \S~8.7 in \cite{K. Ireland}. Part (2) is clear, since $x^k=x$ in $\F_2$.
\end{proof}

\section{Proof of the main theorems and the algorithm}

\subsection{The decomposition formula and its special cases.}\label{subsec:31}

We now prove Theorem A.
\begin{proof}[Proof of Theorem A]
   Note that $\Z/p^a\Z$ has a disjoint decomposition (assuming $p^{a+1}\Z/p^a\Z$ is the empty set)
 \[ \Z/p^a\Z= \bigsqcup_{b=0}^{a} (p^b\Z/p^a\Z-p^{b+1}\Z/p^a\Z). \]
Suppose $\bx=(x_1,\cdots, x_t)\in \Gamma_{J_1}(Q; c,p^a)$, and if $J_1= \emptyset$ and $J_2= I$, suppose $\bx \neq {\bf0}$. Then for $j\in J_2-J_1$,  $x_j\in p^{b_j}\Z/p^a\Z-p^{b_j+1}\Z/p^a\Z$ for some $0\leq b_j\leq a$. Set $b_j=0$ for $j\notin J_2-J_1$. Let $\bb=\bb(\bx)=(b_j)_{j=1,\cdots, t}\in B(J_2,J_1; a)$ and $J_\bb\neq \emptyset$.

For $j\in J_2\cap J_\bb$, the element $\tilde{x}_j=x_j/p^{b_j}$ is a well defined element in $(\Z/p^{a-b_j})^{\times}$. Let $C_j=\{x\in (\Z/p^{a})^{\times}\mid x\equiv \tilde{x}_j\bmod{p^{a-b_j}}\}$. For $j\in J_\bb-J_2$, let $C_j=\{x_j\}$. Then
 \[ C_{\bx}=\prod_{j\in J_\bb} C_j\subseteq  \Gamma_{J_2\cap J_\bb}(Q_\bb; c, p^a). \]
On the other hand, if $Q_\bb\neq 0$, then $J_\bb$ as the set of $j$'s such that $x_j$ appears in $Q_\bb$ is not empty. For $(y_j)_{j\in J_\bb}\in  \Gamma_{J_2\cap J_\bb}(Q_\bb; c, p^a)$, let $\tilde{x}_j=y_j\bmod{p^{a-b_j}}$, then $x_j=p^{b_j} \tilde{x}_j$ is a well defined element in $p^{b_j}\Z/p^a\Z-p^{b_j+1}\Z/p^a\Z$. Let $x_j=0$ for $j\notin J_\bb$. Then $\bx=(x_j)\in  \Gamma_{J_1}(Q; c,p^a)$. In this way, one element $\bx$ corresponds exactly to $p^{\sum_{j: b_j<a} b_j}=p^{s(\bb)}$ elements in $\Gamma_{J_2\cap J_\bb}(Q_\bb; c, p^a)$.

If $J_1= \emptyset$ and $J_2= I$, then ${\bf0} \in \Gamma_{J_1}(Q; c,p^a)$ if only if $p^a\mid c$, which is corresponding to the case $\bb=(a,\cdots, a)$ and $Q_\bb=0$.

In conclusion, \eqref{eq:nj} is proved.
 \end{proof}

\noindent
 \textbf{Special cases of the decomposition formula.} We shall use the following special cases in this paper:

 (1) The case $J=J_1\subsetneq I=J_2$. Then
 \begin{equation} \label{eq:nj1} N_J(Q;  c, p^a)=\sum_{\bb\in B(J,I; a)} p^{-s(\bb)}N^*(Q_\bb; c, p^a).
 \end{equation}
 This means that if we can determine $N^*(Q_\bb; c, p^a)$ for all $\bb\in B(J,I;a)$, then we get $N_J(Q; c, p^a)$.

(2) The case $a=1$. Then
  \begin{equation} \label{eq:nj2} N_{J_1}(Q; c, p)=\sum_{T\subseteq J_2-J_1} N_{J_2-T}(Q_{I-T}; c, p).
 \end{equation}

By the Inclusion-Exclusion Principle, \eqref{eq:nj2} has the following inverse formula
 \begin{equation} \label{eq:nj3}
  N_{J_2}(Q;c,p)=\sum\limits_{T\subseteq  J_2-J_1}(-1)^{|T|}N_{J_1}(Q_{I-T};c,p).
\end{equation}
Take $J_1=\emptyset$ and $J_2=J$  in \eqref{eq:nj3}, then we have
 \begin{equation} \label{eq:nj6}
  N_{J}(Q;c,p)=\sum\limits_{T\subseteq  J}(-1)^{|T|}N (Q_{I-T};c,p).
\end{equation}
This means that $N_J(Q;c,p)$ is determined by $N(Q_{I-T};c,p)$ for all $T\subseteq J$.

\begin{remark} Another interesting question is  to count the number $N_{J_1, J_2}(Q;c,n)$ of solutions of $Q(x_1,\cdots, x_t)\equiv c\bmod{n}$ such that $x_i\in (\Z/n\Z)^\times$ for $i\in J_1$ and  $x_i\notin (\Z/n\Z)^\times$ for $i\in J_2$. First one must keep in mind that no product formula exists in general for $N_{J_1, J_2}(Q;c,n)$ if $J_2\neq \emptyset$. However, by the Inclusion-Exclusion Principle, we have
  \begin{equation} \label{eq:nj1j2}
 N_{J_1, J_2}(Q;  c, n)= \sum_{T\subseteq J_2} (-1)^{|T|} N_{J_1\cup T}(Q;  c, n).
 \end{equation}
 As a consequence,  the values $N_J(Q;c,n)$ for all $J$ determine $N_{J_1, J_2}(Q;c,n)$ for all disjoint pairs $(J_1, J_2)$.
 \end{remark}

\subsection{The lifting formula}

We need the following lemma whose proof is an easy exercise of Newton's Binomial Theorem and $p$-adic analysis:
\begin{lemma} \label{lemma:easy} (1) Let $p$ be an odd prime. For integers $x, y$, $k\geq 1$, and $m\geq 1$, we have
 \[ (x+p^m y)^k-x^k \equiv k x^{k-1}y p^m\bmod{p^{m+v_p(k)+1}}. \]

(2) For  integers $x$ and integer $y$, $k\geq 1$, and $m\geq 1$, then
 \[ (x+2^my)^k-x^k\equiv \begin{cases} 0 \bmod{2^{v_2(k)+2}},&\text{ if $k$ even and $m=1$},\\ k x^{k-1} y\cdot 2^m\bmod{2^{v_2(k)+m+1}},& \text{otherwise.} \end{cases} \]

For odd integer $x$,
  \[ v_2(x^k-1)\geq \begin{cases} 1, &\text{if $k$ odd},\\ 2+v_2(k), &\text{if $k$ even}. \end{cases} \]

(3) Let $U_{p,a}^{(i)}=\{1+p^ix\mid x\in \Z/p^a\Z\}\subseteq (\Z/p^a\Z)^\times$. Then for $f>0$, $(U_{p,a}^{(i)})^{p^f}=U_{p,a}^{(f+i)}$ if $(p,i)\neq (2,1)$ and $(U_{2,a}^{(1)})^{p^f}=U_{2,a}^{(f+2)}$.
\end{lemma}

We are now ready to prove Theorem B and Theorem C.

\begin{proof}[Proof of Theorem B] Write $d=d_p$.
Let $\psi_{a,b}$ be the natural reduction map from $\Gamma_J(c, p^a)$ to $\Gamma_J(c, p^{b})$.

(1) First assume $p$ is odd. Suppose that $j$ satisfies $v_p(\lambda_j k_j)=e_j+f_j=d_j<a$. By Lemma~\ref{lemma:easy}(1), if $(x_1,\cdots, x_j,\cdots, x_t)\in \Gamma_J(c, p^a)$, then $(x_1,\cdots, x_j+p^{a-d_j}y_j,\cdots, x_t)\in \Gamma_J(c, p^a)$ for any $y_i\in \Z/p^a\Z$.

If $a>d_{p,J}$, then $a>d_j+1$ for some $j\in J$. Let $(a_1,\cdots, a_t)\in \Gamma_J(c, p^{a-1})$. Let $u\in \{0,\cdots, p-1\}$. Let $x_i\in \Z/p^a\Z$ be any lifting of $a_i$. Then
 \[ Q(x_1,\cdots, x_j+up^{a-d_j-1},\cdots, x_t)\equiv Q(x_1,\cdots, x_t)+ \frac{\lambda_j k_j}{p^{d_j}} x_i^{k_i-1} u p^{a-1}\bmod{p^a}. \]
Thus there exists exactly one $u\in \{0,\cdots, p-1\}$ such that $(x_1,\cdots, x_j+up^{a-d_j-1},\cdots, x_t)\in \Gamma_J(c, p^a)$, and $\psi_{a,a-1}$ is a $p^{t-1}$-to-$1$ map. Thus we have the lifting formula
 \begin{equation} N_J(c, p^a)= p^{t-1} N_J(c, p^{a-1}) \end{equation}
for all $a>d_{p,J}$.

Now assume $p=2$.  Assume $a>d_{2,J}$.  Then the assumption means that $a>d_j+2$ for some $j\in J$ with $k_j$ even or $a>d_j+1$ for some $j\in J$ with $k_j$ odd. Let $(a_1,\cdots, a_t)\in \Gamma_J(c, 2^{a-1})$.  Let $x_i\in \Z/2^a\Z$ be any lift of $a_i$. Then
 \[ Q(x_1,\cdots, x_j+2^{a-d_j-1},\cdots, x_t)\equiv Q(x_1,\cdots, x_j)+  2^{a-1}\bmod{2^a}. \]
Thus one of $(x_1,\cdots, x_t)$ and $(x_1,\cdots, x_j+2^{a-d_j-1},\cdots, x_t)$ is a solution of $Q(x_1,\cdots, x_t)\equiv c\bmod{n}$, and $\psi_{a,a-1}$ is a $2^{t-1}$-to-$1$ map. Again we have the lifting formula.

(2) Assume  $a\leq d=d_{p,I}$. Suppose $(a_1,\cdots, a_t)\in \F_p^t$, let $\alpha_i\in \Z$ be any lifting of $a_i$. Then
 \[ \lambda_i\alpha_i^{k_i}\equiv (\lambda_i\alpha_i+py_i)^{k_i} \bmod{p^{a}} \]
for any $y_i\in \Z$, and  $Q(\alpha_1,\cdots, \alpha_t)\bmod p^a$ is a fixed element in $\Z/p^a\Z$ independent of the lifting, so the map $\varphi_a$ is well-defined. Thus for $(a_1,\cdots, a_t)\in \Gamma_J(c,p)\subseteq \F_p^t$,
 \[ \# \psi^{-1}_{a,1}(a_1,\cdots, a_t)=\begin{cases} p^{(a-1)t}, & \text{if}\ \varphi_a(a_1,\cdots a_t)= c\bmod{p^a};\\ 0, & \text{if otherwise.} \end{cases} \]

Assume furthermore that $p=2$. For $T\subseteq I$, let $e_T=(e_{T,i})_{i\in I}$ be the element in $\F^t_2$ that $e_{T,i}=1$ for $i\in T$ and $e_{T,i}=0$ for $i\notin T$. Then $\Gamma_J(c,2)$ consists of elements $e_T$ satisfying $T\supseteq J$ and $v_2(\sum_{i\in T}\lambda_T-c)\geq 1$. Let $0$ and $1$ in $\Z$ be the liftings of $0$ and $1$ in $\F_2$ respectively. Then $\varphi_a(e_T)=\sum\limits_{i\in T} \lambda_i\bmod{2^a}$. This finishes the proof of Theorem B(2).
\end{proof}

\begin{corollary} Given the polynomial $Q(x_1,\cdots, x_t)$. If at prime $p$ one has $d_p\geq t$. Then there exists $c\in \Z$ such that $N^*(Q; c,p^{d_p})=0$.
\end{corollary}
\begin{proof} This is because there are $p^{d_p}$ conjugacy classes modulo $p^{d_p}$ but there are only $(p-1)^t$ points in $\F_p^{\times\, t}$.
\end{proof}

\begin{proof}[Proof of Theorem C] Write $k_i=p^{f_i} k'$ such that $(p,k')=1$. By Lemma~\ref{lemma:easy}, if $f\geq f_p$, then for any $i\in I$,  $1+p^f x=(1+py_i)^{p^{f_i}}$ for some $f_i\in \Z/p^a\Z$. If $a\leq c_p+f$, the formula is certainly true. For $a>c_p+f$,  let $u_i, v_i\in \Z$ such that $u_i k'_i+p^{a-k_i} v_i=1$, then $1+p^f x=((1+py_i)^{u_i k_i}=\beta_i^{k_i}$ for some $\beta_i\in (\Z/p^a\Z)^\times$. Thus we have a one-to-one correspondence
 \[ \Gamma_J(c, p^a)\rightarrow \Gamma_J(c(1+p^f x), p^a),\qquad (x_i)\mapsto (x_i\beta_i) \]
and hence $N_J(c, p^a)=N_J(c(1+p^f x), p^a)$.

Now consider the natural map $\psi_{a+1,a}: (\Z/p^{a+1}\Z)^t\rightarrow (\Z/p^{a}\Z)^t$. For $a>c_p+f_p$, $\psi^{-1}_{a+1,a}(\Gamma_J(c, p^a))$ is the disjoint union of $\Gamma_J(c+up^a, p^{a+1})$ for $u\in \{0,\cdots, p-1\}$, but all
$\Gamma_J(c+up^a, p^{a+1})$ are of the same cardinality $N_J(c, p^{a+1})$, hence the lifting formula $N_J(c, p^{a+1})=p^{t-1} N_J(c, p^a)$ holds. This finishes the proof of Theorem C.
\end{proof}

\subsection{An algorithm to compute $N_J(Q; c, p^a)$ if $p\nmid\prod\limits_{i\in I} k_i$}
By Theorems A, B and C, we then have the following algorithm to effectively compute $N_J(Q; c, p^a)$.

\begin{enumerate}
 \item Reduce $Q$ to the reduced form at $p$ (i.e., $d_p(Q)=1$) by Proposition~\ref{prop:easy123}. We suppose $Q$ is reduced hereafter.

\item Compute $N(Q; c, p)$ for all $Q$ by using formulas in Theorem~\ref{theorem:power3}.
\item For $J$ nonempty, compute $N_J(Q;c, p)$ by the Inclusion-Exclusion Principle formula \eqref{eq:nj6}. If $d_{p,J}=1$, use the relation $N_J(Q;c, p^a)=p^{(a-1)(t-1)} N_J(Q; c, p)$ by Theorem B to get $N_J(Q; c, p^a)$, in particular, get  $N^*(Q; c, p^a)$.

\item  For $J$ nonempty and $d_{p,J}=b+1>1$, use  the decomposition formula \eqref{eq:nj1} to compute $N_J(Q;c, p^{a})$ for all $1<a\leq b+1$, then $N_J(Q;c, p^a)=p^{(a-b-1)(t-1)} N_J(Q; c, p^{b+1})$ for $a\geq b+1$ by Theorem B. (Note: the assumption $p\nmid \prod k_i$ means the reduced form of $Q_\bb$ for any $\bb$ in the right hand side of \eqref{eq:nj1} is of depth $1$, hence $N^*(Q_\bb; c, p^a)$  can be computed as in the previous step.)

\item If $c\neq 0$, let $c_p=v_p(c)$. Compute $N(Q; c, p^a)=N(Q; 0, p^a)$ for $a\leq c_p$ and $N(Q; c, p^{c_p+1})$ by the decomposition formula~\eqref{eq:nj1}. Then for $a>c_p+1$, $N(Q;c, p^a) = p^{(a-c_p-1)(t-1)} N(Q; c, p^{c_p+1})$ from Theorem C.

\item Use the decomposition formula~\eqref{eq:nj1} to compute $N(Q; 0, p^a)$ for any given $a$.
\end{enumerate}
\begin{remark} We see that except the last step to compute the case $J=\emptyset$  and $c=0$, the number of steps to compute $N_J(Q;c, p^a)$ is bounded by a constant independent of $a$.

In the case $J$ is nonempty, let $|J|=s$. If $c_p=v_p(c)<b$, by Theorem C, one can furthermore get
 \[ N_J(Q; c, p^{b+1})=p^{b-c_p+c_p s}(p-1)^{s} N(Q_{I-J}; c, p^{c_p+1}). \]
 In particular, if $p\nmid c$, i.e., $c_p=0$, then we just need formulas for $N(Q_{I-J};c, p)$ in Theorem~\ref{theorem:power3} to get $N_J(Q; c, p^a)$.
\end{remark}

\section{Applications of the main theorems}
In this section, we shall apply the general formulas obtained in the previous section to compute $N_J(Q; c, p^a)$ in many special cases. Without loss of generality, we  assume $Q$ is reduced, i.e., $p\nmid \lambda_i$ for some $i$ because of \eqref{eq:aca2}.

\subsection{The linear case $Q(x_1,\cdots, x_t)=\sum\limits_{i=1}^t \lambda_i x_i$.} Consider the linear congruence equation
\begin{align*}
  \lambda_1x_1 + \cdots + \lambda_tx_t \equiv c \ \bmod{p^a}.
\end{align*}
\begin{theorem} \label{theo:linear} Suppose $p\nmid \lambda_i$ for some $i\in I$. For any subset $J$ of $I$ and prime $p$, let $s=\# J$ and $s_p=\# J_p$ where $J_p= \{j\in J\mid p\nmid \lambda_j\}$. Then

(1) The lifting formula holds for all $a\geq 1$:
 \begin{equation}  \label{eq:linear0} N_J(Q; c, p^a)=p^{(a-1)(t-1)} N_J(Q; c, p). \end{equation}

(2) If there exists $i\notin J$, $p\nmid \lambda_i$, then
  \begin{equation} \label{eq:linear1} N_J(Q;c,p)=(p-1)^{s}\ p^{(t-s-1)}; \end{equation}
if for all $i\notin J$, $p\mid \lambda_i$, then
 \begin{equation} \label{eq:linear2} N_J(Q;c,p)=(p-1)^{s} \ p^{(t-s-1)}+(-1)^{s_p}(p-1)^{s-s_p}\ p^{(t-s-1)}(p\delta_c-1)
 \end{equation}
where $\delta_c=1$ if $p\mid c$ and $=0$ if $p\nmid c$.
\end{theorem}
\begin{proof} If there exists $i\notin J$, $p\nmid \lambda_i$, then one can choose all possible $x_j$ for $j\neq i$, and then $x_i$ is decided by the $x_j$'s, so $N_J(Q;c,p^a)=p^{a(t-s-1)}\cdot \varphi(p^a)^s$. Thus \eqref{eq:linear1} holds, so does \eqref{eq:linear0} in this situation.

If for all $i\notin J$, $p\mid \lambda_i$, then there exists $i\in J$ such that $p\nmid \lambda_i$, so $d_{p,J}=1$ and  \eqref{eq:linear0} holds in this situation by Theorem B. Now  one easily has $N_J(Q;c,p)=p^{t-s}(p-1)^{s-s_p} N^*(Q_{J_p}; c, p)$, and by \eqref{eq:nj3},
 \[ \begin{split} N^*(Q_{J_p}; c, p)=&\sum_{i=0}^{s_p-1} (-1)^i \binom{s_p}{i}p^{s_p-i-1}+(-1)^{s_p}\delta_c\\ =&\dfrac{1}{p}(p-1)^{s_p}+(-1)^{s_p}(\delta_c-\frac{1}{p}). \end{split} \]
The theorem is proved.
\end{proof}

\subsection{The  case $Q(x_1,\cdots, x_t)=\sum\limits_{i=1}^t \lambda_i x_i^k$.}
In this subsection, we consider the congruence equation
\begin{align*}
  \lambda_1x_1^k+ \cdots + \lambda_tx_t^k\equiv c \ \bmod{p^a}.
\end{align*}
\subsubsection{A general result.}
The following Theorem is a more detailed version of our algorithm:
\begin{theorem} \label{theorem:kkk1} Suppose prime $p\nmid k$ and $Q$ is reduced at $p$. For $c\neq 0$, let $c_p$ be the $p$-adic valuation of $c$. Let $I_p=\{i\in I\mid p\nmid \lambda_i\}$ and $t_p=\# I_p$. For $J$ a nonempty subset of $I$, let $J_p=\{i\in J\mid p\nmid \lambda_i\}$, $s=\# J$ and $s_p=\# J_p$.
Then

(1) For $c\neq 0$, $N(Q; c, p^a)$ for all $a\geq 1$ is completely determined by $N(Q; 0, p^a)$  for $1\leq a\leq c_p$ and $N(Q; c, p^{c_p+1})$ through the formula
 \begin{equation} \label{eq:kkk1} N(Q;c,p^a)=p^{(a-c_p-1)(t-1)} N(Q; c,p^{c_p+1}),\quad \text{if}\ a\geq c_p+1.
   \end{equation}

 In particular, if $p\nmid c$, then for $a\geq 1$,
  \begin{equation} \label{eq:kkk2} N(Q;c,p^a)=p^{(a-1)(t-1)} N(Q; c,p)= p^{at-a-t_p+1 }N(Q_{I_p}; c,p)
   \end{equation}
where  $N(Q_{I_p}; c,p)$ can be computed by the formulas in Theorem~\ref{theorem:power3}.

(2) If $J_p\neq \emptyset$, i.e., $s_p\neq 0$ and $d_{p,J}=1$, then for any $a\geq 1$, for any $c\in\Z$,
\begin{equation} \label{eq:kkk3}
   N_J(Q;c,p^a)=p^{(a-1)(t-1)} N_J(Q; c, p),
 \end{equation}
\begin{equation} \label{eq:kkk4}
   N_J(Q; c, p)=(p-1)^{s-s_p} p^{t-s+s_p-t_p} \cdot  N_{J_p} (Q_{I_p}; c,p),
\end{equation}
and
\begin{equation} \label{eq:kkk5}
 N_{J_p}(Q_{I_p};c,p)=\sum\limits_{{I_p-J_p\subseteq T \subseteq I_p}}(-1)^{t_p-|T|}N(Q_{T}; c,p)
\end{equation}
where $N(Q_T; c,p)$ can be computed by the formula in Theorem~\ref{theorem:power3}.

In particular,  $N^*(Q; c, p^a)$ can be computable by the formulas above, in this case $J=I$ and $J_p=I_p$.

 (3) If  $d_{p,J}=b+1> 1$, i.e., $s_p=0$, then for $c\in \Z$,
 \begin{equation} N_J(Q; c,p^a)=p^{(a-b-1)(t-1)} N_J(Q; c,p^{b+1}).
 \end{equation}
 If moreover,  $c_p <b$, then
\begin{equation} \label{eq:njquad}
   N_J(Q; c,p^a)=\begin{cases}
    (p-1)^{s} p^{as-s} N(Q_{I-J}; c, p^a), & \text{if $a<c_p+1$}; \\
    (p-1)^{s} p^{(a-c_p-1)(t-1)+c_p s} N(Q_{I-J}; c,p^{c_p+1}), & \text{if $a\geq c_p+1$}. \end{cases}
\end{equation}
Here $N_J(Q; c, p^a)$ for $a\leq b+1$ and $N(Q_{I-J}; c, p^a)$ for $a\leq c_p+1$ can be computed by the decomposition formula \eqref{eq:nj1}.

In particular, if $p\nmid c$, then for $a\geq 1$,
 \begin{equation} N_J(Q; c,p^a)=
    (p-1)^{s} p^{at-a-s-t_p+1} N(Q_{I_p}; c,p) \end{equation}
 where $N(Q_{I_p}; c,p)$ can be computed by Theorem~\ref{theorem:power3}.
\end{theorem}

\subsubsection{The quadratic case.}
In this case, we recall the following well-known result:
 \begin{proposition} Suppose $Q(x_1,\cdots, x_t)=\lambda_1 x_1^2 +\cdots +\lambda_t x_t^2$. For odd prime $p$,
let $\bigl(\frac{\cdot}{p}\bigr)$ be the Legendre symbol. If $p\nmid \prod\limits_{i=1}^t \lambda_i$, then
 \begin{equation}\label{eq:quad3} N(Q;c,p)=\begin{cases}
 p^{t-1}+\big(\frac{c\lambda_1\cdots\lambda_t}{p}\big) {\big(\frac{-1}{p}\big)}^{\frac{t-1}{2}}p^{\frac{t-1}{2}}, & \text{if $t$ odd}; \\
 p^{t-1}-\frac{1}{p}\big(\frac{\lambda_1\cdots\lambda_t}{p}\big){\big(\frac{-1}{p}\big)}^{\frac{t}{2}}p^{\frac{t}{2}}, & \text{if $t$ even and $p\nmid c$};\\
 p^{t-1}+\frac{p-1}{p}\big(\frac{\lambda_1\cdots\lambda_t}{p}\big){\big(\frac{-1}{p}\big)}^{\frac{t}{2}}p^{\frac{t}{2}}, & \text{if $t$ even and $p\mid c$}.\end{cases}
\end{equation}
\end{proposition}
\begin{proof} This  follows from \S8.6 in \cite{K. Ireland}, and can also be found in \cite{agoh}.
\end{proof}
\begin{remark} The above formula holds for $I=\emptyset$. In this case $t=0$ and $N(0;c,p)=1$ if $p\mid c$ and $0$ if not.
\end{remark}
\begin{theorem} \label{theo:quad} Suppose $Q(x_1,\cdots, x_t)=\lambda_1 x_1^2 +\cdots +\lambda_t x_t^2$ and $p\nmid \lambda_i$ for some $i\in I$.

(1) For $p$ odd, suppose $p\nmid \lambda_i$ for some $i\in I$.
Let $I_p=\{p\in I\mid p\nmid \lambda_i\}$, let $t_p=\# I_p$ and $r_p=\#\{i\in I\mid \lambda_i\ \text{is a quadratic non-residue modulo}\ p\}$. Write $p^*=p\cdot\left(\frac{-1}{p}\right)$, and for $i\geq j\geq 0$, write
 \[ A_p(i,j)=\frac{(\sqrt{p^*}+1)^{i-j}(\sqrt{p^*}-1)^{j} + (\sqrt{p^*}-1)^{i-j}(\sqrt{p^*}+1)^{j}}{2}, \]
 \[ B_p(i,j)=\frac{(\sqrt{p^*}+1)^{i-j}(\sqrt{p^*}-1)^{j} - (\sqrt{p^*}-1)^{i-j}(\sqrt{p^*}+1)^{j}}{2}. \]
Then for $a\geq 1$, we have
 \begin{equation} N^*(Q; c, p^a)=p^{(t-1)(a-1)} (p-1)^{t-t_p} N^*(Q_{I_p}; c, p), \end{equation}
where $N^*(Q_{I_p}; c, p)$ is given by
 \begin{equation}   \frac{1}{p}(p-1)^{t_p}+ \begin{cases} (-1)^{r_P} \left(\dfrac{A_p(t_p,r_p)}{\sqrt{p^*}}\bigl(\dfrac{c}{p}\bigr)+\dfrac{B_p(t_p,r_p)}{p}\right), & \text{if $2\nmid t_p$ and $p\nmid c$}; \\
  (-1)^{r_p-1}\left(\dfrac{A_p(t_p,r_p)}{p}+ \dfrac{B_p(t_p,r_p)}{\sqrt{p^*}}\bigl(\dfrac{c}{p}\bigr)\right), & \text{if $2\mid t_p$ and $p\nmid c$};\\
  (-1)^{r_p-1} \dfrac{(p-1)B_p(t_p,r_p)}{p},  & \text{if $2\nmid t_p$ and $p\mid c$};\\
  (-1)^{r_p}\dfrac{(p-1)A_p(t_p,r_p)}{p}, & \text{if $2\mid t_p$ and $p\mid c$}.\end{cases}
 \end{equation}

(2) Moreover, for $J\subseteq I$ such that $d_{p,J}=1$, i.e., if there exists $i\in J$ such
that $p\nmid \lambda_j$.
Let $J_p=\{p\in I\mid p\nmid \lambda_i\}$, let $s=\#J$, $s_p=\# J_p$ and $r_{p,J}=\#\{i\in J\mid \lambda_i\ \text{is a quadratic non-residue modulo}\ p\}$. Then for $a\geq 1$, we have
 \begin{equation} \label{eq:quadnj1} N_J(Q; c, p^a)=p^{(t-1)(a-1)} p^{t-t_p-s+s_{p}}(p-1)^{s-s_{p}} N_{J_p}(Q_{I_p}; c, p), \end{equation}
where
 \begin{equation} \label{eq:quadnj2}  \begin{split} N_{J_p}(Q_{I_p}; c, p)=&(p-1)^{s_{p}} p^{t_p-s_{p}-1} +(-1)^{r_p} (\sqrt{p^*})^{t_p-s_p}\\ \times & \begin{cases} \left(\dfrac{A_p(s_p,r_{p,J})}{\sqrt{p^*}}\bigl(\dfrac{c}{p}\bigr)+\dfrac{B_{p}(s_p,r_{p,J})}{p}\right), & \text{if $2\nmid t_p$ and $p\nmid c$}; \\
   \left(-\dfrac{A_{p}(s_p, r_{p,J})}{p}- \dfrac{B_{p}(s_p,r_{p,J})}{\sqrt{p^*}}\bigl(\dfrac{c}{p}\bigr)\right), & \text{if $2\mid t_p$ and $p\nmid c$};\\ \dfrac{(1-p)B_{p}(s_p,r_{p,J})}{p},  & \text{if $2\nmid t_p$ and $p\mid c$};\\
   \dfrac{(p-1)A_{p}(s_p,r_{p,J})}{p}, & \text{if $2\mid t_p$ and $p\mid c$}.\end{cases} \end{split}
 \end{equation}

(3) For $p=2$, for $J\subseteq I$ such that $d_{2,J}=3$, i.e. if there exists $j\in J$ such that $2\nmid \lambda_j$, then for $a\geq 3$,
 \begin{equation}
 N_J(Q; c, 2^a)=2^{(t-1)(a-3)} N_J(Q; 2, 8);
 \end{equation}
and for $1\leq a\leq 3$,
 \begin{equation}
 N_J(Q; c, 2^a)=2^{(a-1)t}\cdot \#\{J\subseteq T\subseteq I\mid v_2(\sum\limits_{i\in T}\lambda_i-c)\geq a\}.
 \end{equation}

 In particular, for $J=I$, let $c'_2=v_2(\sum\limits_{i\in I}\lambda_i-c)$. Then
  \begin{equation} N^*(Q; c, 2^a)=\begin{cases} 2^{at-a-t+3}, & \text{if}\ a\geq 3\ \text{and}\ c'_2\geq 3;\\ 2^{(a-1)t}, &\text{if}\ a\leq 3\ \text{and}\ c'_2\geq a;\\ 0, &\text{in other cases}.
  \end{cases}
  \end{equation}
\end{theorem}
\begin{remark} For general $Q$ (reduced or not), if we replace the assumption $p\nmid \lambda_i$ for some $i\in J$ by the assumption $\min\{v_p(\lambda_i)\mid i\in I\}= \min\{v_p(\lambda_i)\mid i\in J\}$, along with Proposition~\ref{prop:easy123}(2), we get the formula for $N_J(Q; c, p^a)$ for all $c\in \Z$ and $a\geq 1$.
\end{remark}

\begin{proof} Part (3) follows from Theorem~B(2),  Part (1) is a special case of (2), and \eqref{eq:quadnj1} follows from Theorem~B(1), we just need to prove  \eqref{eq:quadnj2} in Part (2).

By the Inclusion-Exclusion principle, we know
 \[ N_{J_p}(Q_{I_p};c,p)=\sum_{T\subseteq J_p} (-1)^{|T|} N(Q_{I_p-T}; c, p). \]
We use \eqref{eq:quad3} and the above formula to compute $N_{J_p}(Q_{I_p};c,p)$. We compute the main term and the error term separately. The main term is
 \[ \sum_{T\subseteq J_p} (-1)^{|T|} p^{t_p-|T|-1}= (p-1)^{s_{p}} p^{t_p-s_{p}-1}. \]
For the error term, we need the following identities
 \[ \sum_{i\ even} \binom{n}{i} x^i =\frac{(1+x)^n+(1-x)^n}{2}, \]
 \[ \sum_{i\ odd} \binom{n}{i} x^i =\frac{(1+x)^n-(1-x)^n}{2}. \]
In the case $t_p$ is odd and $p\nmid c$, for the subset $T$ of even order,  suppose there are $i$ quadratic residues in $\{\lambda_m\mid m\in T\}$ and $j$ quadratic non-residues, the contribution of the error term in $N(Q_{I_p-T};c, p)$ is
 \[ (-1)^{r_{p}} \left(\frac{c}{p}\right) (\sqrt{p^*})^{t_p-1}\times  (-1)^j (\sqrt{p^*})^{-i-j}. \]
So the contribution for all $T$ of even order is  $(-1)^{r_{p}} \left(\frac{c}{p}\right) (\sqrt{p^*})^{t_p-1}\times $
 \[ \begin{split} & \sum_{i+j\ even} \binom{s_p-r_{p,J}}{i} \binom{r_{p,J}}{j} (-1)^j (\sqrt{p^*})^{-i-j}\\ =& \sum_{i\ even} \binom{s_p-r_{p,J}}{i} (\sqrt{p^*})^{-i}  \sum_{j\ even} \binom{r_{p,J}}{j} (\sqrt{p^*})^{-j}\\ & -\sum_{i\ odd} \binom{s_p-r_{p,J}}{i} (\sqrt{p^*})^{-i}  \sum_{j\ odd} \binom{r_{p,J}}{j} (\sqrt{p^*})^{-j}, \end{split} \]
which is
 \[  (-1)^{r_{p}} (\sqrt{p^*})^{t_p-s_p-1}\left(\frac{c}{p}\right) A_p(s_p, r_{p,J}). \]
Similarly for all $T$ of odd order, the error term contribution is
 \[ \begin{split} &\frac{(-1)^{r_{p}} }{p} (\sqrt{p^*})^{t_p} \sum_{i+j\ odd} \binom{s_p-r_{p,J}}{i} \binom{r_{p,J}}{j} (-1)^j (\sqrt{p^*})^{-i-j}\\ =& (-1)^{r_{p}} (\sqrt{p^*})^{t_p-s_p}\frac{B_p(s_p, r_{p,J})}{p}. \end{split} \]
The other three cases in \eqref{eq:quadnj2} are obtained by the same method.
\end{proof}

\subsubsection{The case $t=2$ and $p\nmid k$.}

For this case, note that if $p\nmid \lambda_1$, let $\lambda_1^{-1}$ be the $p$-adic inverse of $\lambda_1$, then
 \[ N_J(\lambda_1x^k_1+\lambda_2 x_2^k; c, p^a)= N_J(x^k_1+\lambda^{-1}_1\lambda_2 x_2^k; \lambda^{-1}_1c, p^a). \]
Thus we may assume
 \[ Q(x_1, x_2)=x^k_1+\lambda p^e x_2^k \]
such that $p\nmid \lambda$ and $e\geq 0$. We want to compute $N_J(c, p^a)$ for $J=\emptyset,\ \{1\}, \{2\}$ and $I=\{1,2\}$, $c\in \Z$ and $a\geq 1$.

If $p\nmid c$ and $e=0$, by Theorem~\ref{theorem:power3} and note that $J_0(\chi,\chi^{-1})=(p-1)\chi(-1)$ if $\chi\neq 1$, $=p$ if $\chi= 1$, then
 \begin{align} \label{eq:ncpt2} N(c,p)&=p+\sum\limits_{\substack{\chi_1, \chi_2\\ \chi_i^k=1, \chi_i \neq 1}}\chi_1\chi_2(c)\chi_2(\lambda^{-1})J(\chi_1,\chi_2), \\ \label{eq:n0pt2}
 N(0,p)&=1+(p-1)\sum\limits_{\chi:\,\chi^k=1}\chi(-\lambda).
 \end{align}
For $J=\{1\}$ or $I$, then $d_{p,J}=1$. By Theorem B, we have $N_J(c, p^a)=p^{a-1} N_J(c, p)$. Then by \eqref{eq:nj6}, we have

 \begin{proposition} Let $Q(x_1, x_2)=x^k_1+\lambda p^e x_2^k$
such that $p\nmid \lambda k$ and $e\geq 0$. Then
 \begin{equation}
 N_{\{1\}}(c, p^a)=\begin{cases} p^{a-1}(N(c,p)-\sum\limits_{\chi:\,\chi^k=1} \chi(\lambda^{-1}c)), & \text{if} \ e=0\ \text{and}\ p\nmid c; \\ p^a\cdot\sum\limits_{\chi:\,\chi^k=1} \chi(c), & \text{if} \ e\geq 1\ \text{and}\ p\nmid c;\\
 p^{a-1}(N(0,p)-1), & \text{if}  \ e= 0\ \text{and}\ p\mid c;
 \\ 0, & \text{if}  \ e\geq 1\ \text{and}\ p\mid c. \end{cases} \end{equation}
  \begin{equation}
 N^*(c, p^a)=\begin{cases} \label{eq:n*pt2}
 p^{a-1}(N(c,p)-\sum\limits_{\chi:\, \chi^k=1} (\chi(c)+ \chi(\lambda^{-1}c))), & \text{if} \ e=0\ \text{and}\ p\nmid c; \\ p^{a-1}(p-1)\sum\limits_{\chi:\, \chi^k=1} \chi(c), & \text{if} \ e\geq 1\ \text{and}\ p\nmid c;\\
 p^{a-1}(N(0,p)-1), & \text{if}  \ e= 0\ \text{and}\ p\mid c;\\
 0, & \text{if}  \ e\geq 1\ \text{and}\ p\mid c. \end{cases} \end{equation}
Here $N(c,p)$ and $N(0,p)$ are given by \eqref{eq:ncpt2} and \eqref{eq:n0pt2} respectively.
 \end{proposition}
 \begin{remark} In the quadratic case, Theorem~\ref{theo:quad} gives more precise formulas for the cases $J=\{1\}$ or $I$, or $J=\{2\}$ and $e=0$.
 \end{remark}

For $J=\emptyset$ and $\{2\}$, the situation for $N_J(c, p^a)$ is much more complicated.  We first have
\begin{proposition} Let $Q(x_1, x_2)=x^k_1+\lambda p^e x_2^k$
such that $p\nmid \lambda k$ and $e\geq 0$. For $c\neq 0$, let $c_p$ be the $p$-adic valuation of $c$ and $c'= c/p^{c_p}$. For $c=0$, let $c_p=+\infty$. Let $J=\{2\}$ or $\emptyset$. Then
\begin{enumerate}
 \item $N_J(Q;c, p^a)=p^{a-c_p-1} N_J(Q; c, p^{c_p+1})$ for $c\neq 0$.
 \item If $e\geq a$, then $N_{\{2\}}(Q; c, p^a)=p^{a-1}(p-1)N(x_1^k; c, p^a)$ and $N(Q; c, p^a)=p^{a}N(x_1^k; c, p^a)$, and
  \begin{equation}
  N(x_1^k; c, p^a)=\begin{cases} p^{a-\lceil\frac{a}{k}\rceil}, & \text{if}\ c_p\geq a;\\ p^{c_p-\frac{c_p}{k}}\sum\limits_{\chi:\ \chi^k=1} \chi(c'), & \text{if}\ k\mid c_p<a;\\ 0, & \text{if}\ k\nmid c_p<a. \end{cases}
  \end{equation}
  Here $\lceil x\rceil$ meanings the smallest integer $\geq x$.

 \item  If $e<a$,  $N_{\{2\}}(Q;c, p^a)=p^{a-e-1} N_{\{2\}}(Q; c, p^{e+1})$.
\end{enumerate}
\noindent Consequently, the study of $N_J(Q;c, p^a)$ for the set $J=\emptyset$ and $\{2\}$ is reduced to the study of $N(Q; up^{a}, p^{a+1})$ for  $u\in \{0, \cdots,p-1\}$ and $e\leq a$, and $N_{\{2\}}(Q;  up^{e}, p^{e+1})$ for $u\in \{0,\cdots, p-1\}$.
\end{proposition}
\begin{proof} Part (1) follows from Theorem C and Part (3) follows from Theorem B. The first half of (2) follows from Proposition~\ref{prop:easy123}(1).  For the second half of (2), the solutions of $x_1^k\equiv 0\pmod{p^a}$ are of the form $x_1=p^{\lceil\frac{a}{k}\rceil} x_1'$ for $x_1'$ arbitrary. If $c_p<a$, then $x_1^k\equiv c\pmod{p^a}$ is solvable only if $k\mid c_p$, in this case
 \[ N(x_1^k; c, p^a)= p^{c_p-c_p/k}N^*(x^k; c', p^{a-c_p})=p^{c_p-c_p/k}N^*(x^k; c', p), \]
but $N^*(x^k; c', p)=N(x^k; c', p)=\sum\limits_{\chi:\ \chi^k=1} \chi(c')$.
\end{proof}

For the quadratic case, we have
\begin{proposition} \label{prop:4.7}
Let $Q(x_1,x_2)=x_1^2+\lambda p^e x_2^2$ such that $p\nmid 2\lambda$.  Then

(1) For $u\in \{1,\cdots, p-1\}$,
 \begin{equation} N_{\{2\}}(up^e,p^{e+1})=\begin{cases}
 p^{\frac{3e+1}{2}}(1+\bigl(\frac{\lambda u}{p}\bigr)), & \text{if} \ 2 \nmid e; \\
 p^{\frac{3e}{2}}(p-\bigl(\frac{-\lambda}{p}\bigr)-\bigl(\frac{u}{p}\bigr)-1), & \text{if} \ 2 \mid e.
\end{cases}
\end{equation}
For $u=0$,
\begin{equation} N_{\{2\}}(0,p^{e+1})=\begin{cases}
0, & \text{if} \ 2 \nmid e; \\
p^{\frac{3e}{2}}(p-1)(1+\bigl(\frac{-\lambda}{p}\bigr)), & \text{if} \ 2 \mid e.
\end{cases}
\end{equation}

(2) For $u\in \{1,\cdots, p-1\}$ and  $a \geq e$,
\begin{equation}  N(up^{a},p^{a+1})=p^{\frac{2a+e}{2}}\cdot  \begin{cases}
\sqrt{p}(1+\bigl(\frac{u}{p}\bigr)), & \text{if} \ 2 \nmid e \ and \  2 \mid a ; \\
\sqrt{p}(1+\bigl(\frac{\lambda u}{p}\bigr)), & \text{if} \ 2 \nmid e \ and \  2 \nmid a ; \\
(\frac{(a-e)(p-1)}{2}(1+\bigl(\frac{-\lambda}{p}\bigr))+(p-\bigl(\frac{-\lambda}{p}\bigr))), & \text{if} \ 2 \mid e \ and \  2 \mid a ; \\
\dfrac{(a-e+1)(p-1)}{2}(1+\bigl(\frac{-\lambda}{p}\bigr)), & \text{if} \ 2 \mid e \ and \  2 \nmid a.
\end{cases}  \end{equation}
For $e<a$,
 \begin{equation} N(0,p^{a}) =\begin{cases}
p^{\frac{2a+e-1}{2}}, & \text{if} \ 2 \nmid e; \\
p^{\frac{2a+e}{2}}(\frac{(a-e)(p-1)}{2p}(1+\bigl(\frac{-\lambda}{p} \bigr))+1), & \text{if} \ 2 \mid e \ and \  2 \mid a ; \\
p^{\frac{2a+e}{2}}(\frac{(a-e+1)(p-1)}{2p}(1+\bigl(\frac{-\lambda}{p} \bigr))+1), & \text{if} \ 2 \mid e \ and \  2 \nmid a.
\end{cases}  \end{equation}
\end{proposition}
\begin{proof}
We use the decomposition formula in Theorem A to count the number.

(1) Take $J_1=\{2\}$ and $J_2=I$ in Theorem A, then the decomposition formula for $N_{\{2\}}(Q; up^e, p^{e+1})$ is
 \[ N_{\{2\}}(Q; up^e, p^{e+1})=\sum_{j=0}^e p^{-j} N^*(p^{2j}x_1^2+\lambda p^e x_2^2; up^e, p^{e+1})+N^*(\lambda p^e x_2^2; up^e, p^{e+1}). \]
If $j<e/2$, $N^*(p^{2j}x_1^2+\lambda p^e x_2^2; up^e, p^{e+1})=0$. If $j>e/2$,
 \[  \begin{split} N^*(p^{2j}x_1^2+\lambda p^e x_2^2; up^e, p^{e+1})=& p^e(p-1) N^*(\lambda p^e x_2^2; up^e, p^{e+1})\\ =& p^{2e}(p-1) (1+\bigl(\frac{\lambda u}{p}\bigr)). \end{split} \]
 If $j=e/2$, then
 \[ N^*(p^{2j}x_1^2+\lambda p^e x_2^2; up^e, p^{e+1})=p^{2e} (p-2-\bigl(\frac{-\lambda}{p}\bigr)-\bigl(\frac{u}{p}\bigr)
 -\bigl(\frac{\lambda u}{p}\bigr)). \]
Combine the results we get the formula for $N_{\{2\}}(Q; up^e, p^{e+1})$.

The decomposition formula for $N_{\{2\}}(Q; 0, p^{e+1})$ is
 \[ N_{\{2\}}(Q; 0, p^{e+1})=\sum_{j=0}^{e} p^{-j} N^*(p^{2j}x_1^2+\lambda p^e x_2^2; 0, p^{e+1})+N^*(\lambda p^e x_2^2; 0, p^{e+1}). \]
If $j\neq e/2$, $N^*(p^{2j}x_1^2+\lambda p^e x_2^2; 0, p^{a})=0$ and $N^*(\lambda p^e x_2^2; 0, p^{e+1})=0$; for $j=e/2$, $N^*(p^{2j}x_1^2+\lambda p^e x_2^2; 0, p^{e+1})=p^{2e} (p-1)(1+\bigl(\frac{-\lambda}{p}\bigr))$.
So we get the formula for $N_{\{2\}}(Q; 0, p^{e+1})$.

(2) Take $J_1=\emptyset$ and $J_2=\{2\}$, then the decomposition formula for $N(Q; up^a, p^{a+1})$ is
 \[ N(Q; up^a, p^{a+1})=\sum_{j=0}^a p^{-j}  N_{\{2\}}(x_1^2+\lambda p^{e+2j}x_2^2; up^a, p^{a+1})+ N(x_1^2; up^a, p^{a+1}). \]
If $j\geq (a+1-e)/2$, then
 \[ N_{\{2\}}(x_1^2+\lambda p^{e+2j}; up^a, p^{a+1})=p^a(p-1)N(x_1^2; up^a, p^{a+1}), \]
and $N(x_1^2; up^a, p^{a+1})=p^{a/2}(1+\bigl(\frac{u}{p}\bigr))$ if $2\mid a$ and $0$ if $2\nmid a$, so
 \[ \begin{split} \sum_{j\geq (a+1-e)/2} p^{-j} N_{\{2\}}(x_1^2&+\lambda p^{e+2jx_2^2}; up^a, p^{a+1})+ N(x_1^2; up^a, p^{a+1}) \\= &\begin{cases}p^{\frac{3a}{2}+1-\lceil \frac{a+1-e}{2}\rceil} (1+\left(\frac{u}{p}\right)), & \text{if}\ 2\mid a,\\ 0, & \text{if}\ 2\nmid a.    \end{cases} \end{split} \]
If $j< (a-e)/2$, then
 \[ N_{\{2\}}(x_1^2+\lambda p^{e+2j}x_2^2; up^a, p^{a+1})=p^{a-e-2j} N_{\{2\}}(x_1^2+\lambda p^{e+2j}x_2^2; 0, p^{e+2j+1}). \]
If $j=(a-e)/2$, then
 \[ N_{\{2\}}(x_1^2+\lambda p^{e+2j}x_2^2; up^a, p^{a+1})=N_{\{2\}}(x_1^2+\lambda p^{a}x_2^2; up^a, p^{a+1}). \]
We now can just use results in (1) to obtain the formula for $N(Q; up^a, p^{a+1})$.

The decomposition formula for $N(Q; 0, p^{a})$ is
 \[ N(Q; 0, p^{a})=\sum_{j=0}^{a-1} p^{-j} N_{\{2\}}(x_1^2+\lambda p^{e+2j}x_2^2; 0, p^{a})+ N(x_1^2; 0, p^{a}). \]
If $j\geq (a-e)/2$, then
 \[  \begin{split} N_{\{2\}}(x_1^2+\lambda p^{e+2j}x_2^2; 0, p^{a})=& p^{a-1}(p-1)N(x_1^2; 0, p^{a})\\
 =& p^{2a-\lceil\frac{a}{2}\rceil-1}(p-1). \end{split} \]
If $j<(a-e)/2$, then
 \[ N_{\{2\}}(x_1^2+\lambda p^{e+2j}x_2^2; 0, p^{a})=p^{a-e-2j-1} N_{\{2\}}(x_1^2+\lambda p^{e+2j}; 0, p^{e+2j+1}) \]
which is given by formulas in (1). Combine these results, we get the formula for $N(Q; 0, p^{a})$.
\end{proof}

\begin{remark}
For completeness, let us study $N_J(Q; c,2^a)$ for $Q(x_1, x_2)=x_1^2+ 2^e \lambda x_2^2$ and $2\nmid \lambda$. The cases $J=\{1\}$ and  $\{1,2\}$  are given in part (3) of Theorem~\ref{theo:quad}. Here we give steps to compute  $N_J(Qc,2^a)$ for $J=\{2\}$ or  $\emptyset$.

(1) We first compute $N(x_1^2; c, 2^a)$. Assume that $c=2^{c_2}u$  with $u$ odd for $c\neq 0$. Then
\begin{itemize}
\item if $c=0$ or $c_2\geq a$,  $N(x_1^2;0,2^a)=2^{a-\lceil\frac{a}{2}\rceil}$;
\item if $a\geq c_2+3$, $N(x_1^2; c,2^a)=N(x_1^2; c,2^{c_2+3})$ (by Theorem C);
 \item if $c_2+1\leq a \leq c_2+3$, $N(x_1^2; c,2^a)=
2^{a-\frac{c_2}{2}-1}$ if $2\mid c_2$ and $ u\equiv 1 \pmod{ 2^{a-c_2}}$ or $0$ if otherwise.
\end{itemize}

(2) For $J=\{2\}$, if $a> e+3$, by Theorem B, we have
\[N_{\{2\}}(Q; c,2^a)=2^{a-e-3}N_{\{2\}}(Q; c, 2^{e+3}).\]
If $a \leq e+3$, since $2^ex_2^2\equiv 2^e\pmod{2^a}$ for any $x_2 \in (\Z/2^a\Z)^\times$,
\[N_{\{2\}}(Q; c,2^a)= 2^{a-1}N(x_1^2; c-  2^{e}\lambda, 2^a)\]
with $N(x_1^2; c- 2^{e}\lambda, 2^a)$ be given in part (1).

(3) For $J=\emptyset$, by the decomposition formula in Theorem A, we have
\[ N(Q; c, 2^{a})=\sum_{j=0}^{a-1} 2^{-j} N_{\{2\}}(x_1^2+\lambda 2^{e+2j}x_2^2; c, 2^{a})+ N(x_1^2; c, 2^{a}), \]
where $N_{\{2\}}(x_1^2+\lambda p^{e+2j}x_2^2; c, 2^{a})$ is given in part (2) and $N(x_1^2; c, 2^{a})$ is given in part (1).
\end{remark}

For the general case, we have

\begin{proposition} Let $Q(x_1, x_2)=x^k_1+\lambda p^e x_2^k$
such that $p\nmid \lambda k$ and $e\geq 0$.  Let $C=N(x^k_1+\lambda x_2^k; u,p)$ and $C_0^*=N(x^k_1+\lambda x_2^k; 0,p)-1$ given by \eqref{eq:ncpt2} and \eqref{eq:n0pt2} respectively. Then

(1) For $u\in \{1,\cdots, p-1\}$,
 \begin{equation} \label{eq:47} N_{\{2\}}(up^e,p^{e+1})=\begin{cases}
 p^{2e-[\frac{e}{k}]}\sum \chi(u), & \text{if} \ k \nmid e;\\
 p^{\frac{(2k-1)e}{k}}(C-\sum \chi(u)), & \text{if} \ k \mid e.
\end{cases}
\end{equation}
For $u=0$,
\begin{equation} \label{eq:48} N_{\{2\}}(0,p^{e+1})=\begin{cases}
0, & \text{if} \ k \nmid e; \\
p^{\frac{(2k-1)e}{k}}C^*_0, & \text{if} \ k \mid e.
\end{cases}
\end{equation}

(2) For $u\in \{1,\cdots, p-1\}$ and  $a \geq e$,
\begin{equation}  N(up^{a},p^{a+1})=\begin{cases}
p^{\frac{2a(k-1)+e}{k}}C+\frac{p^{\frac{2a(k-1)+e}{k}}
-p^{\frac{ak+e(k-1)}{k}}}{p^{k-2}-1}C_0^*, & \text{if} \ k \mid e \ and \  k \mid a ; \\
p^{\frac{ak+e(k-1)}{k}}\cdot\frac{p^{(k-2)\lceil
\frac{a-e}{k}\rceil}-1}{p^{k-2}-1}C^*_0, & \text{if} \ k \mid e \ and \  k \nmid a ; \\
p^{\frac{(2k-1)a}{k}-[\frac{a-e}{k}]}\sum \chi(u), & \text{if} \ k \nmid e \ and \  k \mid a; \\
p^{\frac{(2k-1)a+e}{k}-[\frac{a}{k}]}\sum \chi(u), & \text{if} \ k \nmid e \ and \  k \mid a-e ; \\
0, & otherwise.
\end{cases}  \end{equation}
For $e<a$,
 \begin{equation} N(0,p^{a}) =\begin{cases}
p^{2a-\lceil\frac{a-e}{k}\rceil+\lceil\frac{a}{k}\rceil}+
p^{a+e-1-\frac{e}{k}}\cdot
\frac{p^{(k-2)\lceil\frac{a-e}{k}\rceil}-1}{p^{k-2}-1}C^*_0, & \text{if} \ k \mid e; \\
p^{2a-\lceil\frac{a-e}{k}\rceil+\lceil\frac{a}{k}\rceil}, & \text{if} \ k \nmid e.
\end{cases}  \end{equation}
Here the sum $\sum$ is over all characters $\chi$ such that $\chi^k=1$, and $[n]$ means the largest integer $\leq n$.
\end{proposition}
\begin{proof}
The proof of part (1) is similar to the proof of  Proposition~\ref{prop:4.7}. We just show how to get the formulas of part (2).

Take $J_1=\emptyset$ and $J_2=\{2\}$, then the decomposition formula for $N(Q; up^a, p^{a+1})$ is
 \[ N(Q; up^a, p^{a+1})=\sum_{j=0}^a p^{-j}  N_{\{2\}}(x_1^k+\lambda p^{e+kj}x_2^k; up^a, p^{a+1})+ N(x_1^k; up^a, p^{a+1}). \]
If $e+kj> a$, i.e. $j\geq [\frac{a-e}{k}]+1$, then
 \[N_{\{2\}}(x_1^k+\lambda p^{e+kj}; up^a, p^{a+1})=p^{a}(p-1)N(x_1^k; up^a, p^{a+1}), \]
and $N(x_1^k; up^a, p^{a+1})=p^{a-\frac{a}{k}}\sum \chi(u)$ if $k\mid a$ and $0$ if $k\nmid a$, so
 \[ \begin{split} \sum_{j=[\frac{a-e}{k}]+1}^a p^{-j} N_{\{2\}}(x_1^k&+\lambda p^{e+kj}x_2^k; up^a, p^{a+1})+ N(x_1^k; up^a, p^{a+1}) \\= &\begin{cases}p^{2a-\frac{a}{k}-[\frac{a-e}{k}]}\sum \chi(u), & \text{if}\ k\mid a;\\ 0, & \text{if}\ k \nmid a.    \end{cases} \end{split} \]
If $e+kj< a$, i.e. $j\leq \lceil\frac{a-e}{k}\rceil-1$, then
 \[ N_{\{2\}}(x_1^k+\lambda p^{e+kj}x_2^k; up^a, p^{a+1})=p^{a-e-kj} N_{\{2\}}(x_1^k+\lambda p^{e+kj}x_2^k; 0, p^{e+kj+1}). \]
By \eqref{eq:48}, we have
 \[ \begin{split} \sum_{j=0}^{\lceil\frac{a-e}{k}\rceil-1} &p^{-j} N_{\{2\}}(x_1^k+\lambda p^{e+kj}x_2^k; 0, p^{a})\\= &\begin{cases}p^{a+e-\frac{e}{k}}\cdot\frac{p^{(k-2)\lceil\frac{a-e}{k}\rceil}-1}{p^{k-2}-1}C_0^*, & \text{if}\ k\mid e; \\ 0, & \text{if}\ k \nmid e.    \end{cases} \end{split} \]
If $e+kj=a$, i.e $j=\frac{a-e}{k}$, then by \eqref{eq:47} we have
 \[ \begin{split} p^{-j} N_{\{2\}}(&x_1^k+\lambda p^{a}x_2^k; up^a, p^{a+1})\\=
 &\begin{cases}p^{2a-\frac{a-e}{k}-\frac{a}{k}}(C-\sum \chi(u)), & \text{if}\ k\mid a \ and \ k\mid a-e; \\
 p^{2a-\frac{a-e}{k}-[\frac{a}{k}]}\sum \chi(u), & \text{if}\ k\nmid a \ and \ k\mid a-e.   \end{cases} \end{split} \]
Thus we get the formula for $N(Q; up^{a}, p^{a+1})$.

The decomposition formula for $N(Q; 0, p^{a})$ is
 \[ N(Q; 0, p^{a})=\sum_{j=0}^{a-1} p^{-j} N_{\{2\}}(x_1^k+\lambda p^{e+kj}x_2^k; 0, p^{a})+ N(x_1^k; 0, p^{a}). \]
If $e+kj\geq a$, i.e. $j\geq \lceil\frac{a-e}{k}\rceil$ then
 \[  \begin{split} N_{\{2\}}(x_1^k+\lambda p^{e+kj}x_2^k; 0, p^{a})=& p^{a-1}(p-1)N(x_1^k; 0, p^{a})\\
 =& p^{2a-\lceil\frac{a}{k}\rceil-1}(p-1). \end{split} \]
 so
\[\sum_{j=\lceil\frac{a-e}{k}\rceil}^{a-1} p^{-j} N_{\{2\}}(x_1^k+\lambda p^{e+kj}x_2^k; 0, p^{a})+ N(x_1^k;0, p^{a})= p^{2a-\lceil\frac{a}{k}\rceil-\lceil\frac{a-e}{k}\rceil}\]
If $e+kj< a$, i.e. $j\leq \lceil\frac{a-e}{k}\rceil-1$, then
 \[ N_{\{2\}}(x_1^k+\lambda p^{e+kj}x_2^k; 0, p^{a})=p^{a-e-2j-1} N_{\{2\}}(x_1^k+\lambda p^{e+kj}x_2^k; 0, p^{e+2j+1}) \]
and
   \[ \begin{split} \sum_{j=0}^{\lceil\frac{a-e}{k}\rceil-1} p^{-j} &N_{\{2\}}(x_1^k+\lambda p^{e+kj}x_2^k; 0, p^{a+1})\\= &\begin{cases}p^{a+e-\frac{e}{k}-1}\cdot\frac{p^{(k-2)\lceil
   \frac{a-e}{k}\rceil}-1}{p^{k-2}-1}C_0^*, & \text{if}\ k\mid e; \\ 0, & \text{if}\ k \nmid e.    \end{cases} \end{split} \]
thus we get the formula for $N(Q; 0, p^{a})$.
\end{proof}

\begin{remark}
The case $t\geq 3$ can also be computed, but the discussion is a little bit tedious.
\end{remark}

\subsection{The example $Q(x_1,x_2,x_3)=9x_1+3x_2^3+x_3^9$ for $p=3$.}
At last we consider the congruence equation
\begin{align*}
  Q(x_1,x_2,x_3)=9x_1+3x_2^3+x_3^9\equiv c \bmod{3^a},  \ (a \geq 3).
\end{align*}
which is not included in the algorithm.

For $c\neq 0$, write $c=3^{c_3}c'$; for $c=0$ set $c_3=+\infty \geq a$ for any $a$. Since for any $J\neq \emptyset$, $d_{3,J}=3$, by Theorem B, we have
 \[ N_J(Q; c, 3^a)=3^{2(a-3)} N_J(Q; c, 27). \]
After simple calculation, we then get $N_J(Q; c,27)$ in Table 1.

\begin{table}[!htbp]
\caption{$N_J(c,27)$ for $J$ nonempty}
 \begin{tabular}{||r||r|r|r|r|r|r|r|r||}
     \hline
     \hline
       $c'$ &0 & 1, & 3, & 9, & 2,4,  & 8,10, & 6,12, & else \\
        &  &  26 & 24  &  18 &  23,25 &  17,19 & 15,21 &  \\
       \hline \hline
       $N^*(c,27)$  & 0 & 0 & 0 & 0 & 0 & 0 & 0 & $3^6$ \\ \hline
       $N_{\{1,2\}}(c,27)$ & 0 & 0 & 0 & 0 & 0 & 0 & $3^6$ & $3^6$ \\ \hline
       $N_{\{1,3\}}(c,27)$ & 0 & 0 & 0 & 0 & 0 & $3^6$ & 0 & $3^6$ \\ \hline
       $N_{\{2,3\}}(c,27)$ & 0 & 0 & 0 & 0 & $3^6$ & 0 & 0 & $3^6$ \\ \hline
       $N_{\{1\}}(c,27)$ & 0 & 0 & 0 & $3^6$ & 0 & $3^6$ & $3^6$ & $3^6$ \\ \hline
       $N_{\{2\}}(c,27)$ & 0 & 0 & $3^6$ & 0 & $3^6$ & 0 & $3^6$ & $3^6$ \\ \hline
       $N_{\{3\}}(c,27)$ & 0 & $3^6$ & 0 & 0 & $3^6$ & $3^6$ & 0 & $3^6$\\
       \hline
   \end{tabular}
\end{table}

For $J=\emptyset$, the map $\varphi_3: (a_1, a_2, a_3)\mapsto Q(\alpha_1,\alpha_2, \alpha_3) \bmod{27}$ from $(\Z/3\Z)^3$ to $\Z/27\Z$ is found to be one-to-one. Note that any solution $(\beta_1,\beta_2,\beta_3)\in \Gamma(Q;c,27)$ is a lifting of some $(a_1,a_2,a_3)\in \varphi_3^{-1}(c)$, but we always have
\[Q(\beta_1,\beta_2,\beta_3)= \varphi_3(a_1,a_2,a_3).\]
Thus for any $c\in \Z$, we have $N(Q; c, 27)=3^6$. In fact, we have $N(Q; c, 3^a)=3^{2a}$ for $a\leq 3$. For the case $a>3$, we use the notation $N_{J_1,J_2}$ introduced in the remark of \S\ref{subsec:31}, then
\[N(c,3^a)=N_{\emptyset,\{2,3\}}(c,3^a)+N_{\{2\},\{3\}}(c,3^a)
+N_{\{3\},\{2\}}(c,3^a)+N_{\{2,3\}}(c,3^a).\]
We compute the right hand side term by term:
\begin{itemize}
\item if $c_3=0$, then $N_{\emptyset,\{2,3\}}=N_{\{2\},\{3\}}=0$, $N_{\{3\},\{2\}}=3^{2a}$ for $c'\equiv 1,8,10,17,19,26(\bmod 27)$, and $N_{\{2,3\}}=3^{2a}$ for $c'\equiv 2,4,5,7,11,13,\\14,16,20,22,23,25(\bmod 27)$ from Table 1;
\item if $c_3=1$, then $N_{\emptyset,\{2,3\}}=N_{\{3\},\{2\}}=N_{\{2,3\}}=0$, and $N_{\{2\},\{3\}}=3^{2a}$;
\item if $c_3\geq2$, $N_{\{2\},\{3\}}=N_{\{3\},\{2\}}=N_{\{2,3\}}=0$, and $N_{\emptyset,\{2,3\}}=3^{2a}$.
\end{itemize}
Thus we have
 \[ N(Q; c, 3^a)=3^{2a} \]
 for any $a>0$.

 \subsection*{Acknowledgement}
Research is partially
supported by National Key Basic Research Program of China (Grant No. 2013CB834202) and National Natural Science Foundation of China (Grant No. 11571328).

\end{document}